\documentclass[11pt]{article}
\usepackage{amssymb, amsmath, amsthm}
\usepackage{epsfig}
\usepackage{fig4tex}

\newtheorem{teo}{Theorem}
\newtheorem{cor}[teo]{Corollary}
\newtheorem{prop}[teo]{Proposition}
 \theoremstyle{definition}
\newtheorem{dhef}[teo]{Definition} \theoremstyle{remark}
\newtheorem{rk}[teo]{Remark}

\usepackage{times}
\usepackage{graphicx}
\usepackage{amscd}
\usepackage{amsfonts}
\usepackage{latexsym}

\long\def\elimina#1{} 
\def\neweq#1{\begin{equation}\label{#1}}
\def\endeq{\end{equation}}
\DeclareMathOperator{\cotan}{cotan}

\def\integ{\int_\Omega}
\def\eq#1{(\ref{#1})}

\def\R{\mbox{${\rm I}\!{\rm R}$}}

\def\Wp{W_0^{1,p}}
\def\Web{\mathcal{W}_p}
\def\eps{\varepsilon}

\def\ou{{u_p}}

\def\C{\mathcal{C}}
\def\S{\mathcal{S}}
\def\J{J_p}

\def\Om{\Omega}

\def\P{\mathcal{P}}
\def\Wp{W_0^{1,p}}

\def\t{\tau_p}
\def\wt{w_p}

\title{Sharp bounds for the $p$-torsion of convex planar domains}

\author{Ilaria FRAGAL\`A - Filippo GAZZOLA - Jimmy LAMBOLEY}

\date{}

\begin{document}
\baselineskip14pt \maketitle

\begin{abstract}
{\rm We obtain some sharp estimates for the $p$-torsion of convex
planar domains in terms of their area, perimeter, and inradius. The approach we
adopt relies on the use of web functions ({\it i.e.} functions
depending only on the distance from the boundary), and on the
behaviour of the inner parallel sets of convex polygons. As an
application of our isoperimetric inequalities,
we consider the shape optimization problem which consists in
maximizing the $p$-torsion among polygons having a given number of
vertices and a given area. A long-standing conjecture by
P\'olya-Szeg\"o states that the solution is the regular polygon. We
show that such conjecture is true within the subclass of polygons
for which a suitable notion of ``asymmetry measure'' exceeds a
critical threshold.}
\end{abstract}

\noindent {\small {\it 2000MSC\ }: 49K30, 52A10, 49Q10.}


\section{Introduction}\label{intro}

Let $\Omega\subset\R^2$ be an open bounded domain and let $p \in (1,+ \infty)$. Consider the boundary value problem
\begin{equation}\label{euler}
\left\{\begin{array}{ll} -\Delta_p u=1\quad & \mbox{in }\Omega\\
u=0\quad & \mbox{on }\partial\Omega\, ,
\end{array}\right.
\end{equation}
where $\Delta_pu={\rm div}(|\nabla u|^{p-2}\nabla u)$ denotes the
$p$-Laplacian. The {\it $p$-torsion} of $\Omega$ is defined by
\begin{equation}\label{defi} \t (\Omega):= \integ|\nabla\ou|^p = \int_\Omega u_p\ ,\end{equation}
being $u _p$ the unique solution to (\ref{euler}) in $\Wp(\Omega)$. Notice that the second equality in
(\ref{defi}) is obtained by testing  (\ref{euler}) by $u_p$ and integrating by parts.
Since (\ref{euler}) is the Euler-Lagrange equation of the variational problem
\begin{equation}\label{energy}
\min_{u\in \Wp(\Omega)}\J(u)\,, \quad \hbox{ where } \J(u)
=\int_{\Omega} \Big ( \frac{1}{p} |\nabla u| ^p - u \Big ) \, ,
\end{equation}
there holds
$$ \t(\Omega) = \frac{p}{1-p} \min_{u\in \Wp(\Omega)}\J(u)\, .$$
A further characterization of the $p$-torsion  is provided by the
equality $\t(\Omega) = S(\Omega)^ {1/(p-1)}$, where $S(\Omega)$ is
the best constant for the Sobolev inequality $\|u\|_{L^1(\Omega)} ^p
\leq S(\Omega) \| \nabla u \| _{L^p(\Omega)}  ^p$ on
$\Wp(\Omega)$.\par
The purpose of this paper is to provide some sharp bounds for $\t(\Omega)$, holding for a convex planar domain $\Omega$, in terms
of its area, perimeter, and inradius (in the sequel denoted respectively by
$|\Omega|$, $|\partial \Omega|$, and $R_\Omega$). The original motivation for
studying this kind of shape optimization problem draws its origins
in the following long-standing conjecture by P\'olya and Szeg\"o:
\begin{equation}\label{PSconj}
\hbox{{\it  Among polygons with a given area and  $N$ vertices,
the regular $N$-gon maximizes $\t$} \, .}
\end{equation}
A similar conjecture is stated by the same Authors also for the
principal frequency and for the logarithmic capacity, see
\cite{posz}. For $N =3$ and $N= 4$ these conjectures were proved by
P\'olya and Szeg\"o themselves \cite[p.\ 158]{posz}. For $N \geq 5$,
to the best of our knowledge, the unique solved case is the one of
logarithmic capacity, see the beautiful paper \cite{sz} by Solynin
and Zalgaller; the cases of torsion and principal frequency are
currently open. In fact let us remind that, for $N \geq 5$, the
classical tool of Steiner symmetrization fails because it may
increase the number of sides, see \cite[Section 3.3]{h}.\par
The approach we adopt in order to provide upper and lower bounds for the $p$-torsion in terms of
geometric quantities, is based on the idea of considering a proper subspace $\Web(\Omega)$ of $\Wp(\Omega)$
and to address the minimization problem for the functional $\J$ on
$\Web(\Omega)$. More precisely, we consider the subspace of functions depending only on the distance
$d(x)={\rm dist}(x,\partial\Omega)$ from the boundary:
$$\Web(\Omega)=\{u\in \Wp(\Omega)\ : \  u(x)=u(d(x))\}\ .$$
Functions in $\Web(\Omega)$ have the same level lines as $d$, namely
the boundaries of the so-called {\em inner parallel sets}, $\Omega
_t := \{ x \in \Omega \, :\, d(x) > t \}$, which were first used in
variational problems by P\'olya and Szeg\"o \cite[Section 1.29]{posz}.
Later, in \cite{g}, the elements of $\Web (\Omega)$ were called {\em web
functions},  because in case of planar polygons the level lines of
$d$ recall the pattern of a spider web. We refer to \cite{cg,cg1}
for some estimates on the minimizing properties of these functions,
and to the subsequent papers \cite{cfg,cfg1} for their application
in the study of the generalized torsion problem. Actually, the
papers \cite{cfg,cfg1} deal with the problem of estimating how
efficiently $\t (\Omega)$ can be approximated by the {\it web
p-torsion},  defined as
$$\wt (\Omega) := \frac{p}{1-p}\min_{u\in\Web(\Omega)}\J(u)\ .$$
While the value of $\t(\Omega)$ is in general not known (because the
solution to problem (\ref{euler}) cannot be determined except for
some special geometries of $\Omega$), the value of $\wt
(\Omega)$ admits the following explicit expression in terms of the
parallel sets $\Omega _t$:
\begin{equation}\label{ice}
\wt(\Omega) = \int_0 ^ {R_\Omega} \frac{|\Omega _t| ^ q}{|\partial
\Omega _t| ^ {q-1}} \, dt\ ,
\end{equation}
where $q= \frac{p}{p-1}$ is the conjugate exponent
of $p$, and $R_\Omega$ is the inradius of $\Omega$ (see \cite{cfg1}).

Clearly, since $\Web (\Omega) \subset \Wp (\Omega)$, $\wt(\Omega)$
bounds $\t(\Omega)$ from below. On the other hand, when $\Omega$ is
convex, $\t(\Omega)$ can be bounded from above by a constant
multiple of $\wt (\Omega)$, for some constant which tends to $1$ as
$p\to+\infty$. In fact, in \cite{cfg1} it is proved that, for any
$p\in(1, +\infty)$, the following estimates hold and are sharp:
\begin{equation}\label{proven}
\forall\; \Om\in \C, \quad \frac{q+1}{2^q}<\frac{\wt(\Om)}{\t(\Om)}\leq 1
\end{equation}
where $\C$ denotes the class of planar bounded convex domains;
moreover the right inequality holds as an equality if and only if
$\Omega$ is a disk. Note that, if $p\to+\infty$, then $q\to1$ and
the constant in the left hand side of \eq{proven} tends to $1$.\par
In this paper, we prove some geometric estimates for  $\t(\Omega)$
in the class $\C$, which have some implications in the conjecture
(\ref{PSconj}). More precisely, we consider the following shape functionals:
\neweq{shapefunct}
\Omega \mapsto \frac{\t(\Om)|\partial\Om|^q}{|\Om|^{q+1}}\quad\mbox{and}\quad
\Omega \mapsto \frac{\t(\Om)}{R_{\Om}^{q}|\Om|}\ .
\endeq
Let us remark that the above quotients are invariant under dilations and that convex sets which agree
up to rigid motions (translations and rotations) are systematically identified throughout the paper.

Our main results are Theorems \ref{estimate2} and \ref{new}, which give sharp bounds for the functionals
\eq{shapefunct} when $\Omega$ varies in $\C$. We also exhibit minimizing and maximizing sequences.
These bounds are obtained by combining sharp bounds for the web $p$-torsion (see Theorem \ref{estimate} and the second part
of Theorem \ref{new}) with \eq{proven}. As a consequence of our results we obtain the validity of some weak forms of
P\'olya-Szeg\"o conjecture (\ref{PSconj}). On the class $\P $ of
convex polygons we introduce a sort of ``asymmetry measure'' such as
$$\forall \Om\in\P, \quad \gamma  (\Omega) :=  \frac{| \partial \Omega| }{| \partial \Omega ^\circledast|} \in [1, + \infty)\ ,$$
where $\Omega ^ \circledast$ denotes the regular polygon with the
same area and the same number of vertices as $\Omega$.
Then, if the $p$-torsion $\tau_p(\Omega)$ is replaced by the web
$p$-torsion $w_p(\Omega)$, (\ref{PSconj}) holds in the following refined form:
\begin{equation}\label{PSconjweb}
\forall \, \Omega \in \P \, , \quad \wt (\Omega) \leq \gamma
(\Omega) ^ {-q} \wt (\Omega  ^\circledast) \ .
\end{equation}
Consequently, on the class $\P _N$ of convex polygons with $N$
vertices, conjecture (\ref{PSconj}) holds true for those $\Omega$
which are sufficiently ``far'' from $\Omega ^\circledast$, meaning
that $\gamma (\Omega)$ exceeds a threshold depending on $N$ and $p$:
\begin{equation}\label{PSconjnonweb}
\forall \, \Omega \in \P _N\, : \, \gamma (\Omega) \geq \Gamma_{N, p} ,\quad \t (\Omega) <  \t (\Omega ^ \circledast) \ .
\end{equation}
The value of the threshold $\Gamma_{N,p}$ can be explicitly characterized (see Corollary \ref{cor2}) and
tends to $1$ as $p \to + \infty$. \par
The paper is organized as follows. Section \ref{main} contains the
statement of our results, which are proved in Section \ref{proofs}
after giving in Section \ref{geometric} some preliminary material of
geometric nature. Section \ref{open} is devoted to some related open
questions and perspectives.

\section{Results}\label{main}

We introduce the following classes of convex planar domains:\par\smallskip
$\C$ = the class of bounded convex domains in $\R^2$;\par\smallskip
$\C_o$ = the subclass of $\C$ given by tangential bodies to a disk;\par\smallskip
$\P$ = the class of convex polygons;\par\smallskip
$\P_N$ = the class of convex polygons having $N$ vertices ($N\geq3$).\par \smallskip

Tangential bodies to a disk are domains $\Omega \in \C$ such that, for some disk $D$, through each
point of $\partial \Omega$ there exists a tangent line to $\Omega$
which is also tangent to $D$. Domains in $\P\cap\C_o$ are circumscribed polygons, whereas domains in
$\C_o\setminus\P$ can be obtained by removing from a circumscribed
polygon some connected components of the complement (in the polygon
itself) of the inscribed disk. In particular, the disk itself
belongs to $\C_o$.\par
Our first results are the following sharp bounds for the $p$-torsion of
convex planar domains. We recall that, for any given $p\in (1, +
\infty)$, $q:= \frac{p}{p-1}$ denotes its conjugate exponent.

\begin{teo}\label{estimate2}
For any $p \in (1, + \infty)$, it holds
\begin{equation}\label{eq:isoDp}
\forall \Om\in\C,\quad
\frac{1}{q+1}<\frac{\t(\Om)|\partial\Om|^q}{|\Om|^{q+1}}<\frac{2^{q+1}}{(q+2)(q+1)}\ .
\end{equation}
Moreover,\par
$\bullet$ the left inequality holds asymptotically with equality sign for any sequence of thinning rectangles;\par
$\bullet$ the right inequality holds asymptotically with equality sign for any sequence of thinning isosceles triangles.
\end{teo}

By sequence of thinning rectangles or triangles, we mean that the
ratio between their minimal width and diameter tends to 0.
We point out that, in the particular case when $p=2$, the statement
of Theorem \ref{estimate2} is already known. Indeed, the left
inequality  in (\ref{eq:isoDp}) holds true for any simply connected
set in $\R^2$ as discovered by P\'olya \cite{polya}; the right
inequality in (\ref{eq:isoDp}) for convex sets is due to Makai
\cite{makai}, though its method of proof, which is different from
ours, does not allow to obtain the {\it strict} inequality.

Our approach to prove Theorem \ref{estimate2} employs as a major
ingredient the following sharp estimates for the web $p$-torsion of
convex domains, which may have their own interest.

\begin{teo}\label{estimate}
For any $p\in(1,+\infty)$, it holds
\begin{equation}\label{eq:isoNp}
\forall \Om\in\C,\quad
\frac{1}{q+1}<\frac{\wt(\Om)|\partial\Om|^q}{|\Om|^{q+1}}\leq\frac{2}{q+2}\ .
\end{equation}
Moreover,\par
$\bullet$ the left inequality holds asymptotically with equality sign for any sequence of thinning rectangles;\par
$\bullet$ the right inequality holds with equality sign for $\Omega \in\C_o$.
\end{teo}

Let us now discuss the implications of the above results in the shape optimization problem which consists in
maximizing $\t$ in the class of convex polygons with a given area and a given number of vertices:
\begin{equation}\label{probl}
\max \Big \{ \t (\Omega) \ : \ \Omega \in {\mathcal P} _N \, ,\
|\Omega | = m \Big \}\,.
\end{equation}

We recall that, for any $\Omega \in \P$, $\Omega ^ \circledast$
denotes the regular polygon with the same area and the same number
of vertices as $\Omega$. Moreover, we set
$$\forall \Om\in\P, \quad \gamma  (\Omega) :=  \frac{| \partial \Omega| }{| \partial \Omega ^\circledast|} \, ;$$
notice that by the isoperimetric inequality for polygons (see Proposition \ref{prop:isop}), $\gamma(\Om)\in[1,+\infty)$ and $\gamma(\Omega)>1$ if $\Omega \neq \Omega ^\circledast$.
With this notation, it is  straightforward to deduce from Theorem \ref{estimate} the following

\begin{cor}\label{cor1} The regular polygon is the unique maximizer of $\wt$ over polygons in $\P$ with a given area and a given number of vertices.
More precisely, the following refined isoperimetric inequality holds:
\begin{equation}\label{PSconjweb'}
\forall \, \Omega \in \P \, , \quad \wt (\Omega) \leq \gamma
(\Omega) ^ {-q} \wt (\Omega  ^\circledast) \ .
\end{equation}
\end{cor}

As a consequence, using \eqref{proven}, we obtain some information on the shape
optimization problem (\ref{probl}):

\begin{cor}\label{cor2} Let $\displaystyle\Gamma_{N,p}:=\Big(\frac{\wt(\Omega^\circledast)}{\t(\Omega^\circledast)}\Big)^{1/q}\frac{2}{(q+1)^{1/q}}$.
Then,
$$\forall \, \Omega \in \P _N\, , \quad \t (\Omega) <\Gamma_{N, p} ^q \gamma(\Omega) ^ {-q} \t (\Omega ^ \circledast) \ .$$
In particular, the $p$-torsion of the regular $N$-gon is larger than
the $p$-torsion of any polygon in $\P_N$ having the same area and an
asymmetry measure larger than the threshold $\Gamma_{N,p}$:
\begin{equation}\label{PSconjnonweb'}
\forall \, \Omega \in \P _N\, , \quad\, \gamma(\Omega)\geq\Gamma_{N,p}\, \Rightarrow\, \t(\Omega)<\t(\Omega^\circledast)\ .
\end{equation}
\end{cor}

Some comments on Corollary \ref{cor2} are gathered in the next remark.

\begin{rk}
(i) Using again (\ref{proven}) we infer
$$1 <\Gamma_{N,p} < \frac{2}{(q+1) ^ {1/q}}<2\quad\forall N,p\, ,\qquad\lim_{p\to +\infty}\Gamma_{N,p}=1\ .$$
Hence, asymptotically with respect to $p$, the condition
$\gamma(\Omega)\geq\Gamma_{N,p}$ appearing in (\ref{PSconjnonweb'})
becomes not restrictive. Moreover, if $p=2$, we have
$\Gamma_{N,2}\le2/\sqrt{3}\approx1.15$  and the dependence on $N$ of
$\Gamma_{N,2}$ can be enlightened by using the numerical values
given in \cite{cg1}:
\begin{center}
\begin{tabular}{|c|c|c|c|c|c|c|c|c|c|}
\hline
$N$ & 3 & 4 & 5 & 6 & 7 & 8 & 9 & 10 & 20  \\
\hline $\Gamma_{N,2}\approx$ & 1.054 & 1.089 & 1.108 & 1.121 & 1.129
& 1.135 & 1.138 &
1.141 & 1.149 \\
\hline
\end{tabular}
\end{center}

(ii) Though the validity of (\ref{PSconj}) is known for triangles, in order to give an idea of the efficiency of
Corollary \ref{cor2}, consider the case $N=3$ and $p=2$. The equilateral triangle
$$T^\circledast:=\left\{(x,y)\in\R^2;\, y>0\, ,\ -\frac{1}{2}+\frac{y}{\sqrt{3}}<x<\frac{1}{2}-\frac{y}{\sqrt{3}}\right\}$$
satisfies $|T^\circledast|=\frac{\sqrt{3}}{4}$ and $|\partial T^\circledast|=3$. The solution to \eqref{euler} is explicitly given by
$$u(x,y)=\frac{\sqrt{3}}{8}\left(y-\frac{4}{\sqrt{3}}y^2+\frac{4}{3}y^3-4x^2y\right)$$
so that $\tau_2(T^\circledast)=\sqrt{3}/640$. Moreover, by \eqref{eq:energycirc} below we find $w_{2}(T^\circledast)=\sqrt{3}/768$ and, in turn,
that $\Gamma_{3,2}=\sqrt{10}/3\approx1.054$.\par
Consider now the isosceles triangles $T_k$ having the basis of length $k>0$ and the two equal sides of length
$$\ell_k=\sqrt{\frac{3}{4k^2}+\frac{k^2}{4}}\quad\mbox{so that}\quad|\partial T_k|=k+\sqrt{\frac{3}{k^2}+k^2}\quad\mbox{and}\quad|T_k|
=\frac{\sqrt{3}}{4}=|T^\circledast|\, ,$$
(notice that $T_{1}=T^\circledast$). Therefore,
$$\gamma(T_k)=\frac{k+\sqrt{\frac{3}{k^2}+k^2}}{3}$$
and $\gamma(T_k)\ge\Gamma_{3,2}$ if and only if $2\sqrt{10}\,
k^3-10\, k^2+3\ge0$, which approximatively corresponds to $k \not
\in (0.760, 1.301)$.\end{rk}

We conclude this section with a variant of Theorems \ref{estimate2} and \ref{estimate}.

\begin{teo}\label{new} For every $p \in (1, + \infty)$, it holds
\begin{eqnarray}
& \displaystyle{\forall \Om\in\C,\quad
\frac{1}{(q+2)2^{q-1}}\leq\frac{\t(\Om)}{R_{\Om}^{q}|\Om|}<\frac{2^q}{(q+1)^2}}&
\label{new2}
\\ \noalign{\smallskip}
& \displaystyle{\forall \Om\in\C,\quad
\frac{1}{(q+2)2^{q-1}}\leq\frac{\wt(\Om)}{R_{\Om}^{q}|\Om|}<\frac{1}{q+1}}\, .& \label{new1}
\end{eqnarray}
Moreover,\par
$\bullet$ the left inequality in $(\ref{new2})$ holds with equality sign for balls;\par
$\bullet$ the left inequality in $(\ref{new1})$ holds with equality sign for $\Omega \in \C_o$;\par
$\bullet$ the right inequality in $(\ref{new1})$ holds asymptotically with equality sign for a sequence of thinning rectangles.
\end{teo}

The right inequality in (\ref{new2}) {\it is not} sharp. In fact, for $p=2$, one has the sharp inequalities
$$\forall \Om\in\C,\quad \frac{1}{8}\leq\frac{\tau _2(\Om)}{R_{\Om}^{2}|\Om|} \leq \frac{1}{3}\, ,$$
see \cite[p.\ 100]{posz} for the left one, and \cite{makai} for the right one.

Using the isoperimetric inequalities (\ref{new2}) and (\ref{new1}), one can also derive statements similar to Corollaries
\ref{cor1} and \ref{cor2}, where $\gamma(\Om)$ is replaced by another ``asymmetry measure'' given by
$$\widetilde \gamma (\Omega) = \frac{R_{\Omega^\circledast }}{ R_{\Omega}} \ .$$

\section{Geometric preliminaries}\label{geometric}

In this section we present some useful geometric properties of
convex polygons, which will be exploited to prove Theorem
\ref{estimate}. First, we recall an improved form of the
isoperimetric inequality in the class $\P$, whose proof can be found
for instance in \cite[Theorem 2]{cfg}. For any $\Omega \in \P$, we set
\begin{equation}\label{COmega}
C_{\Om}:=\sum_{i}\cotan \frac{\theta_{i}}{2}\, ,\, \ \hbox{ being } \theta_{i}
\hbox{ the inner angles of } \Omega\ .
\end{equation}

\begin{prop}\label{prop:isop}For every $\Om\in\P$, it holds
\begin{equation}\label{eq:iso}
|\Om|\leq\frac{|\partial\Om|^2}{4C_{\Om}}\ ,
\end{equation}
with equality sign if and only if $\Om\in \P \cap \C _o $, namely when $\Omega$ is a circumscribed polygon.
\end{prop}

Next, we recall that, denoting by $R_\Omega$ the inradius of any $\Omega \in \P$, for every $t\in[ 0 , R_\Omega]$, the {\it inner
parallel} sets of $\Omega$ are defined by
$$\Om_{t}:=\{x\in\Om\ :\  {\rm dist}(x,\partial\Om)> t\}$$
(notice in particular that $\Om _{R_\Om} = \emptyset$). Then we focus our attention on the behaviour of the map $t \mapsto
C_{\Omega_t}$ on the interval $[0, R_\Omega]$, and on the related expression of Steiner formulae. For every $\Omega \in \P$, we set

$$r_\Omega:= \sup \big \{ t \in [0, R_\Omega] \ :\ \Omega _t \hbox { has the same number of vertices as $\Omega$ } \big \}\ . $$

Clearly, if $r_\Omega<R_\Omega$, the number of vertices of $\Omega_t$ is strictly less than the number of vertices of
$\Omega$ for every $t \in[r_\Omega,R_\Omega)$.

\begin{prop}\label{boh}
For every $\Omega \in \P$ and $t\geq 0$, $\Om_{t}\in\P$ and the map $t\mapsto C_{\Om_{t}}$ is piecewise constant on $[0, R_\Omega)$.
Moreover, for every $t \in [0, r_\Omega]$, it holds
\begin{equation}\label{steiner}
|\Om_{t}|=|\Om|-|\partial\Om|t+C_{\Om}t^2 \qquad \textrm{ and }\qquad |\partial\Om_{t}|=|\partial\Om|-2C_{\Om}t.
\end{equation}
Finally, for every $t \in [0, R_\Omega]$, it holds
\begin{equation}\label{macrosteiner}
|\partial\Om_{t}|\leq |\partial\Om|-2\pi t.
\end{equation}
\end{prop}
\begin{proof}
For $t$ small enough, the sides of $\Om_{t}$ are parallel and at distance $t$ from the sides of $\Om$, and the corners of $\Om_{t}$
are located on the bisectors of the angles of $\Om$. $r_{\Omega}$ is actually the first time when two of these bisectors intersect at
a point having distance $t$ from at least two sides, see Figure \ref{parallelsets}.

\begin{figure}[ht]
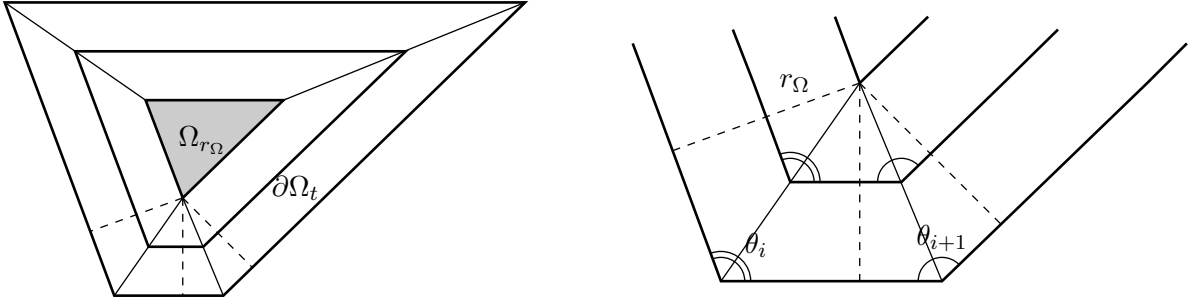

\begin{minipage}[!h]{85mm}

\figinit{0.37pt}
\figpt 0:(0,0) \figpt 1:(-107,0) \figpt 21:(-182,200) \figpt
2:(-70,-100) \figpt 3:(41.5,-100) \figpt 4:(145,0) \figpt
24:(350,200) \figpt 5:(0,-100) \figpt 6:(-94,-34) \figpt 7:(71,-71)

\figpt 11:(-53.5,0) \figpt 12:(-35,-50) \figpt 13:(21,-50) \figpt
14:(72.5,0)\figpt 34:(228,150)
\figpt 31:(-110,150)

\figpt 8:(-12.5,33.5) \figpt 9:(34.5,33.5)
\figpt 44:(103,100)
\figpt 41:(-38,100)

\figpt 50:(20,30)
\figpt 51:(90,10)

\psbeginfig{}

\psset(fillmode=yes,color=0.8) \psline[41,0,44,41]
\psset(fillmode=no,color=0)
\psset(width=1) \psline[21,2,3,24,21] \psline[31,12,13,34,31] \psline[41,0,44,41]

\psset(width=0.5)

\psline[21,41] \psline[24,44]
\psline[0,2] \psline[0,3]

\psset(width=0.5) \psset(dash=8) \psline[0,5] \psline[0,6]
\psline[0,7]

\psendfig
\figvisu{\figBoxA}{
} {
 \figwriten 50:{$\Om_{r_{\Om}}$}(15)
 \figwritee 51:{$\partial\Om_{t}$}(1)
 }
\centerline{\box\figBoxA}
\end{minipage}
\begin{minipage}[!h]{85mm}

\figinit{0.75pt}
\figpt 0:(0,0) \figpt 1:(-107,0) \figpt 21:(-114.5,20) \figpt
2:(-70,-100) \figpt 3:(41.5,-100) \figpt 4:(145,0) \figpt
24:(165,20) \figpt 5:(0,-100) \figpt 6:(-94,-34) \figpt 7:(71,-71)

\figpt 8:(-12.5,33.5) \figpt 9:(34.5,33.5)

\figpt 11:(-53.5,0) \figpt 12:(-35,-50) \figpt 13:(21,-50) \figpt
14:(72.5,0) \figpt 31:(-64,27) \figpt 34:(100,27)

\psbeginfig{}

\psset(width=1) \psline[21,1,2,3,4,24] \psline[31,11,12,13,14,34]
\psline[0,8] \psline[0,9]

\psset(width=0.5)

\psline[0,2] \psline[0,3]

\psset(width=0.5) \psset(dash=8) \psline[0,5] \psline[0,6]
\psline[0,7]

\psset(width=0.5)

\psset(dash=1) \psarccirc 2; 12(0,110) \psarccirc 2; 15(0,110)
\psarccirc 3; 12(45,180)

\psarccirc 12; 12(0,110) \psarccirc 12; 15(0,110) \psarccirc 13;
12(45,180)

\psset(dash=8)

\psendfig
\figvisu{\figBoxA}{
} {
 \figwritene 2:{$\theta_{i}$}(17)
 \figwriten 3:{$\theta_{i+1}$}(15)
 \figwritew 0:{$r_{\Om}$}(25)
 }
\centerline{\box\figBoxA}
\end{minipage}
\vspace{-0.8cm}\centerline{\box\figBoxA} \caption{Intersection of bisectors}
\label{parallelsets}\end{figure}

Therefore, for $t< r_{\Om}$, $\Om_{t}$ has the same angles as $\Om$ (so $C_{\Om_{t}}=C_{\Om}$ by \eq{COmega}),
and we notice that the perimeter of grey areas in Figure \ref{fig:Steiner} is
$2t\cotan(\theta_{i}/2)$, and their areas are $t^2\cotan(\theta_{i}/2)$,
which gives \eqref{steiner} (still valid for $t=r_{\Om}$ by continuity).

\begin{figure}
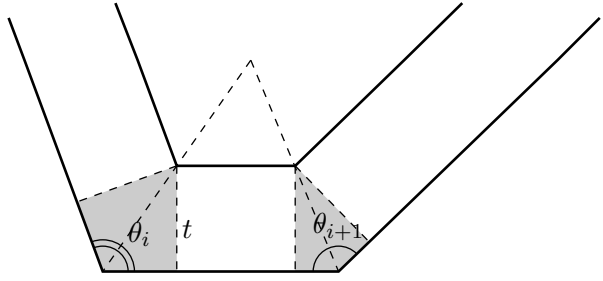

\begin{center}

\figinit{0.8pt}
\figpt 0:(0,0) \figpt 1:(-107,0) \figpt 21:(-114.5,20) \figpt
2:(-70,-100) \figpt 3:(41.5,-100) \figpt 4:(145,0) \figpt
24:(165,20) \figpt 5:(-35,-100) \figpt 15:(21,-100) \figpt
6:(-82,-67) \figpt 7:(56,-86)

\figpt 11:(-53.5,0) \figpt 12:(-35,-50) \figpt 13:(21,-50) \figpt
14:(72.5,0) \figpt 31:(-64,27) \figpt 34:(100,27)


\figpt 20:(-30,-80)

\psbeginfig{}


\psset(fillmode=yes,color=0.8) \psline[12,5,2,6,12]
\psline[13,15,3,7,13] \psset(fillmode=no,color=0)

\psset(width=1) \psline[21,1,2,3,4,24] \psline[31,11,12,13,14,34]

\psset(width=0.5)

\psset(width=0.5) \psset(dash=8) \psline[12,5] \psline[13,15]
\psline[12,6] \psline[13,7] \psline[0,2] \psline[0,3]

\psset(width=0.5)

\psset(dash=1) \psarccirc 2; 12(0,110) \psarccirc 2; 15(0,110)
\psarccirc 3; 12(45,180)

\psset(dash=8)

\psendfig
\figvisu{\figBoxA}{
} {
 \figwritene 2:{$\theta_{i}$}(17)
 \figwriten 3:{$\theta_{i+1}$}(15)
\figwritec [20]{$t$}
 }
\centerline{\box\figBoxA} \vspace{-0.5cm}
\caption{How to derive Steiner formulae}
\label{fig:Steiner}
\end{center}
\vspace{-1cm}
\end{figure}

Let us now show that the map $t\mapsto C_{\Om_{t}}$ is piecewise constant on $[0, R_\Omega)$, assuming that $r_\Om < R_\Om$. Once
$t=r_{\Om}$, $\Om_{t}$ still has sides parallel to the ones of $\Om$
but loses at least one of them. Again, $C_{\Om_{t}}$ is constant for
$t\geq r_{\Om}$ until the next value of $t$ such that another intersection of bisectors appears (we now consider
bisectors of $\Om_{r_{\Om}}$). The number of discontinuities of $t\mapsto C_{\Om_{t}}$ is finite since $\Om$ has a finite number of
sides, and therefore iterating the previous argument, we get that $t\mapsto C_{\Om_{t}}$ is piecewise constant.\par
Finally, from \eqref{COmega} we infer that $C_\Omega\ge\pi$ for any $\Omega\in\P$,
so that \eqref{macrosteiner} follows from the concavity of the map
$t\mapsto|\partial\Omega_t|$ on $[0, R_\Omega]$ (see \cite[Sections
24 and 55]{bonn}).
\end{proof}

A special role is played by polygons $\Omega\in \P$ such that
$r_\Omega=R_\Omega$, namely polygons $\Omega$ whose inner parallel
sets all have the same number of vertices as $\Omega$ itself. These are {\it polygonal stadiums}, characterized by the following

\begin{dhef}\label{def:stadium}  We call $\S$ the class of {\it polygonal stadiums}, namely polygons $P ^ \ell \in \P$
such that there exist a circumscribed polygon $P \in  \P \cap \C_o $
having two parallel sides, and a nonnegative number $\ell$ such
that, by choosing a coordinate system with origin in the center of
the disk inscribed in $P$ and the $x$-axis directed as two parallel
sides of $P$, $P ^ \ell$ can be written as
\begin{equation}\label{deco}
P^\ell:=\Big(P_{-}-\Big(\frac{\ell}{2},0\Big)\Big)\bigcup
\Big(\Big[-\frac{\ell}{2},\frac{\ell}{2}\Big]\times\Big(-R_P,R_P\Big)\Big)\bigcup
\Big(P_{+}+\Big(\frac{\ell}{2},0\Big)\Big)\ , \end{equation} where
$P _-$ (resp. $P _+$) denotes the set of points $(x, y)\in P$ with
$x<0$ (resp. $x>0$), and $R_P$ is the inradius of $P$, see Figure
\ref{fig:special}.
\end{dhef}

\begin{figure}[ht]
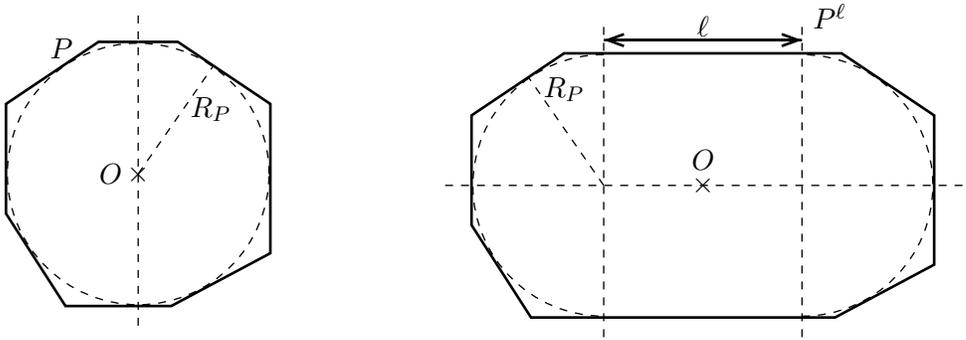

\begin{center}
\begin{minipage}[!h]{55mm}

\figinit{0.5pt}
\figpt 0:(0,0)
\figpt 1:(0,100) \figpt 2:(-30,100) \figpt 3:(-100,53) \figpt
4:(-100,-30) \figpt 5:(-55,-100) \figpt 6:(0,-100)

\figpt 7:(30,100) \figpt 8:(100,53) \figpt 9:(100,-60) \figpt
10:(25,-100) \figpt 11:(0,-100)

\figpt 12:(0,120) \figpt 13:(0,-120)

\figpt 14:(57,82)

\figpt 15:(55,49)
\figpt 16:(-75,79)

\figpt 21:(0,50) \figpt 22:(-15,50) \figpt 23:(-50,26) \figpt
24:(-50,-15) \figpt 25:(-27,-50) \figpt 26:(0,-50)

\figpt 27:(15,50) \figpt 28:(50,26) \figpt 29:(50,-30) \figpt
210:(13,-50) \figpt 211:(0,-50)

\psbeginfig{}

\psset(width=1) \psline[1,2,3,4,5,6] \psline[1,7,8,9,10,11]
\psset(width=0.5)
\psset(dash=8) \psarccirc 0; 99(0,360)
\psline[12,13] \psline[0,14]

\psendfig
\figvisu{\figBoxA}{
} { \figwritec[15]{$R_{P}$}
\figwritene 16:{$P$}(12)
\figwritec[0]{$\times$}
\figwritew 0:{$O$}(12)
}
\centerline{\box\figBoxA}
\end{minipage}
\begin{minipage}[!h]{95mm}
\figinit{0.5pt}

\figpt 0:(0,0)
\figpt 100:(75,0)
\figpt 1:(0,100) \figpt 2:(-30,100) \figpt 3:(-100,53) \figpt
4:(-100,-30) \figpt 5:(-55,-100) \figpt 6:(0,-100)

\figpt 7:(180,100) \figpt 8:(250,53) \figpt 9:(250,-60) \figpt
10:(175,-100) \figpt 11:(150,-100)

\figpt 12:(0,120) \figpt 13:(0,-120)
\figpt 21:(-120,0) \figpt 22:(280,0)

\figpt 212:(150,120) \figpt 213:(150,-120)

\figpt 14:(-57,82)

\figpt 15:(-30,54) \figpt 18:(75,110)
\figpt 17:(-75,79)

\figpt 16:(0,110) \figpt 17:(150,110)

\figpt 20:(150,0)

\psbeginfig{}

\psset(width=1) \psline[1,2,3,4,5,6] \psline[1,7,8,9,10,6]
\psarrow[17,16] \psarrow[16,17] \psset(width=0.5)
\psset(dash=8) \psarccirc 0; 99(90,270)
\psarccirc 20; 99(-90,90)
\psline[12,13]
\psline[21,22]
\psline[212,213] \psline[0,14]

\psendfig
\figvisu{\figBoxA}{
} { \figwriten 15:{$R_{P}$}(10) \figwriten 18:{$\ell$}(3)
\figwritene 17:{$P^\ell$}(12)
\figwritec[100]{$\times$}
\figwriten 100:{$O$}(12)
}
\centerline{\box\figBoxA}
\end{minipage}
\end{center}
\vspace{-1cm}\caption{A circumscribed polygon $P$ and a polygonal stadium $P ^
\ell$} \label{fig:special}
\vspace{-0.3cm}
\end{figure}

\begin{prop}\label{speciali}
Let $\Om \in \P$. There holds $r_\Omega = R_\Omega$  if and only if
$\Om \in \S$.
\end{prop}

\begin{proof}  We use the same notation as in Definition
\ref{def:stadium}. Assume that $\Om= P ^ \ell \in \S$. Then the
bisectors of the angles of $\Om$ intersect either at
$(-\frac{\ell}{2},0)$ or at $(\frac{\ell}{2},0)$, which are at
distance $R_{\Om}$ from the boundary, see Figure \ref{fig:special2}.
In particular, if $\Om$ is circumscribed to a disk, namely if $\ell
= 0$, then the bisectors of the angles of $\Om$ all intersect at the
center of the disk. Therefore, $\Om_{t}$ has the same number of
sides as $\Om$ if $t<R_{\Om}$.

\begin{figure}[ht]
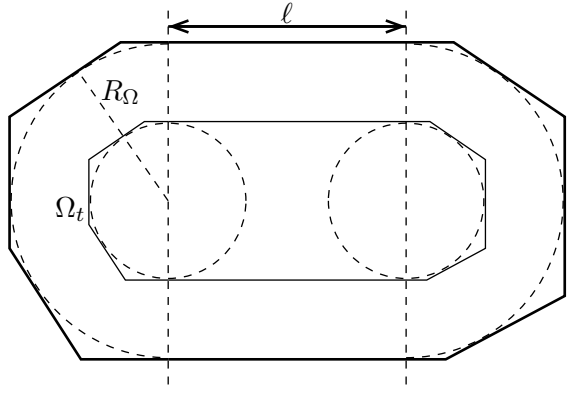

\begin{center}

\begin{minipage}[!h]{95mm}
\figinit{0.6pt}

\figpt 0:(0,0)
\figpt 1:(0,100) \figpt 2:(-30,100) \figpt 3:(-100,53) \figpt
4:(-100,-30) \figpt 5:(-55,-100) \figpt 6:(0,-100)

\figpt 7:(180,100) \figpt 8:(250,53) \figpt 9:(250,-60) \figpt
10:(175,-100) \figpt 11:(150,-100)

\figpt 12:(0,120) \figpt 13:(0,-120)

\figpt 212:(150,120) \figpt 213:(150,-120)

\figpt 14:(-57,82)

\figpt 15:(-30,54) \figpt 18:(75,110)

\figpt 16:(0,110) \figpt 17:(150,110)

\figpt 20:(150,0) \figpt 21:(0,50) \figpt 22:(-15,50) \figpt
23:(-50,26) \figpt 24:(-50,-15) \figpt 25:(-27,-50) \figpt
26:(0,-50)

\figpt 27:(165,50) \figpt 28:(200,26) \figpt 29:(200,-30) \figpt
210:(163,-50) \figpt 211:(150,-50)

\psbeginfig{}

\psset(width=1) \psline[1,2,3,4,5,6] \psline[1,7,8,9,10,6]
\psarrow[17,16] \psarrow[16,17] \psset(width=0.5)
\psline[21,22,23,24,25,26] \psline[21,27,28,29,210,26]
\psset(dash=8) \psarccirc 0; 99(90,270) \psarccirc 20; 99(-90,90)
\psarccirc 20; 49(-180,180) \psarccirc 0; 49(0,360) \psline[12,13]
\psline[212,213] \psline[0,14]

\psendfig
\figvisu{\figBoxA}{
} { \figwriten 15:{$R_{\Om}$}(7) \figwriten 18:{$\ell$}(2)
\figwritenw 24:{$\Om_{t}$}(3) } \centerline{\box\figBoxA}
\end{minipage}
\end{center}
\vspace{-1cm}\caption{Parallel sets of a polygonal stadium $P ^
\ell$} \label{fig:special2}
\end{figure}

Conversely, assume that $R_\Om=r_\Om$. The set $\{x\in\Om\, :\, d(x)=R_\Om\}$ is convex with empty interior, so either it is a point, or a segment.
If it is a point, then its distance to each side is the same, and therefore the disk having this point as a center and radius $R_{\Om}$ is tangent
to every side of $\Om$, so that $\Om$ is circumscribed to a disk. If it is a segment, we choose coordinates such that this segment is
$\left[\left(-\frac{\ell}{2},0\right);\left(\frac{\ell}{2},0\right)\right]$
for some positive number $\ell$. Every point of this segment is at
distance $R_{\Om}$ from the boundary, so $\Om$ contains the
rectangle
$\left(-\frac{\ell}{2},\frac{\ell}{2}\right)\times\left(-R_\Om,R_\Om\right)$.
Considering $$P:=\Big(\Om\cap\Big\{x\leq
-\frac{\ell}{2}\Big\}+\Big(\frac{\ell}{2},0\Big)\Big)\bigcup
\Big(\Om\cap\Big\{x\geq
\frac{\ell}{2}\Big\}+\Big(-\frac{\ell}{2},0\Big)\Big),$$
we have that $P$ is circumscribed and $\Om=P^\ell$.\end{proof}

\begin{rk} Thanks to Proposition \ref{speciali}, for any polygonal stadium $P^\ell$,
the validity of the Steiner formulae (\ref{steiner}) extends
for $t$ ranging over the whole interval $[0, R_{P^\ell}]$. Moreover,
the value of the coefficients $|P^\ell|$, $|\partial P ^\ell|$ and
$C_{P ^ \ell}$ appearing therein, can be expressed only in terms of
$|P|$, $R_P$, and $\ell$ (see Section \ref{proofs}). It is enough to
use the following elementary equalities deriving from decomposition (\ref{deco})
$$ |P ^ \ell| = |P | + 2 \ell R_P\ , \qquad
|\partial P ^ \ell| = |\partial P | + 2 \ell\ , \qquad C_{P^\ell} =
C_P\ , \qquad R_{P^\ell} = R_P\ ,
$$
and the following identities holding  for every $P \in \P \cap \C_o$
\begin{equation}\label{circum}
C_{P}=\frac{|P|}{R_{P}^2} \ , \qquad |\partial P
|=\frac{2|P|}{R_{P}}\ .
\end{equation}
\end{rk}

Finally, we show that the parallel sets of any convex polygon $\Om$
are polygonal stadiums for $t$ sufficiently close to $R_{\Om}$:

\begin{prop}\label{becomestadium}
For every $\Om\in\P$, there exists $\overline t \in [0, R_\Omega)$
such that the parallel sets $\Om_{t}$ belong to $\S$ for every
$t\in\left[\overline{t},R_{\Om}\right)$.
\end{prop}
\begin{proof}
We define $\overline{t}$ as the last time $t<R_{\Om}$ such that
$\Om$ loses a side (we may have $\overline{t}=0$). Therefore
$\forall t\in\left[\overline{t},R_{\Om}\right], \Om_{t}$ has a constant number
of sides, and so is in the class $\S$ by Proposition \ref{speciali}.
\end{proof}

\section{Proofs}\label{proofs}

\subsection{Proof of Theorem \ref{estimate}}\label{ssect:proof1}

We first prove Theorem \ref{estimate} for $\Om\in\P$, then we prove it for all $\Om\in\C$.\par\smallskip
{\it $\bullet$ Step 1: comparison with inner parallel sets.} For a
given $\Om\in\P$, we wish to compare the value of the energy
$\frac{\wt(\Om)|\partial\Om|^q}{|\Om|^{q+1}}$ with the one of its
parallel set $\Om_{\eps}$ for small $\eps$. To that aim, we use the
representation formula (\ref{ice}) for $\wt(\Omega)$, and Steiner's
formulae \eqref{steiner}. In applying them we recall that, by
Proposition \ref{boh} the map $t\mapsto C_{\Omega_t}$ is piecewise
constant for $t\in[0,R_\Omega)$, and in particular it equals
$C_\Omega$ on $[0, r_\Omega]$. Taking also into account that
$(\Om_{\eps})_{t}=\Om_{\eps+t}$, as $\eps\to0$ we have
\begin{eqnarray}
\frac{\wt(\Om_{\eps})|\partial\Om_{\eps}|^q}{|\Om_{\eps}|^{q+1}}&=&
\frac{\int_{0}^{R_{\Om}-\eps}\frac{|(\Om_{\eps})_{t}|^q}{|\partial(\Om_{\eps})_{t}|^{q-1}}dt\, |\partial\Om_{\eps}|^q}{|\Om_{\eps}|^{q+1}} \label{discuss}\\
&=&\frac{\left[\wt(\Om)-\int_{0}^{\eps}\frac{|\Om_{t}|^q}{|\partial\Om_{t}|^{q-1}}dt\right]\left[|\partial\Om|-2C_{\Om}\, \eps\right]^q}
{\left[|\Om|-|\partial\Om|\, \eps\right]^{q+1}}+o(\eps), \notag\\
&=&\frac{|\partial\Om|^q}{|\Om|^{q+1}}\left[\wt(\Om)-\frac{|\Om|^q}{|\partial\Om|^{q-1}}\eps\right]
\left[1-2q\frac{C_{\Om}}{|\partial\Om|}\eps\right]\left[1+(q+1)\frac{|\partial\Om|}{|\Om|}\eps\right]+o(\eps), \notag\\
&=&\frac{\wt(\Om)|\partial\Om|^q}{|\Om|^{q+1}}+\left[(q+1)\frac{|\partial\Om|^{q+1}}{|\Om|^{q+2}}\wt(\Om)-
\frac{|\partial\Om|}{|\Om|}-2q\frac{C_{\Om}\wt(\Om)|\partial\Om|^{q-1}}{|\Om|^{q+1}}\right]\eps+o(\eps), \notag\\ \notag
\end{eqnarray}
so that
\begin{equation}\label{comparison}
\frac{\wt(\Om_{\eps})|\partial\Om_{\eps}|^q}{|\Om_{\eps}|^{q+1}}-
\frac{\wt(\Om)|\partial\Om|^q}{|\Om|^{q+1}}=\left[(q+1)\frac{|\partial\Om|^{q+1}}{|\Om|^{q+2}}\wt(\Om)-
\frac{|\partial\Om|}{|\Om|}-2q\frac{C_{\Om}\wt(\Om)|\partial\Om|^{q-1}}{|\Om|^{q+1}}\right]\eps+o(\eps)\, .
\end{equation}
As we shall see in the next steps, formula \eqref{comparison} will
enable us to reach a contradiction if \eqref{eq:isoNp} fails.\par\smallskip\noindent
{\it $\bullet$ Step 2: if \eqref{eq:isoNp} fails for some convex polygon then it also fails
for a polygonal stadium.} Let $\Omega\in\P \setminus \S$, and assume
that \eqref{eq:isoNp} fails. We have to distinguish two cases.\par
{\em First case:} Assume that
\begin{equation}\label{contra1}
\frac{\wt(\Om)|\partial\Om|^q}{|\Om|^{q+1}}>\frac{2}{q+2}.
\end{equation}
Using the isoperimetric inequality \eqref{eq:iso} and (\ref{contra1}), one gets
$$\left[(q+1)\frac{|\partial\Om|^{q+1}}{|\Om|^{q+2}}\wt(\Om)-\frac{|\partial\Om|}{|\Om|}-
2q\frac{C_{\Om}\wt(\Om)|\partial\Om|^{q-1}}{|\Om|^{q+1}}\right]\ge
\frac{q+2}{2}\frac{|\partial\Om|}{|\Om|}\left[\frac{\wt(\Om)|\partial\Om|^q}{|\Om|^{q+1}}-\frac{2}{q+2}\right]>0\, .$$
Inserting this information into \eqref{comparison} shows that
$$\frac{\wt(\Om_{\eps})|\partial\Om_{\eps}|^q}{|\Om_{\eps}|^{q+1}}-\frac{\wt(\Om)|\partial\Om|^q}{|\Om|^{q+1}}>0$$
for sufficiently small $\eps$. In fact, more can be said. By
Proposition \ref{boh} we know that $C_{\Omega_t}=C_\Omega$ for all
$t\in[0,r_\Omega)$. By extending the above argument to all such $t$,
we obtain that, if \eqref{contra1} holds, then the map
$t\mapsto\frac{\wt(\Om_t)|\partial\Om_t|^q}{|\Om_t|^{q+1}}$ is
strictly increasing for $t\in[0,r_\Omega)$. In particular, by
\eqref{contra1},
$$
\frac{\wt(\Om_{\eps})|\partial\Om_{\eps}|^q}{|\Om_{\eps}|^{q+1}}>\frac{\wt(\Om)|\partial\Om|^q}{|\Om|^{q+1}}>\frac{2}{q+2}
\qquad \forall \eps \in (0, r_\Omega].
$$
So, if $\Om_{r_\Om} \in \S$ , we are done since it violates
\eqref{eq:isoNp}. At $t=r_\Om$ the number of sides of $\Omega_t$
varies. If $\Om_{r_\Om} \notin\S$, we repeat the previous argument
to the next interval where $C_{\Omega_t}$ remains constant. Again,
the map $t\mapsto\frac{\wt(\Om_t)|\partial\Om_t|^q}{|\Om_t|^{q+1}}$
is strictly increasing on such interval. In view of Proposition
\ref{becomestadium}, this procedure enables us to obtain some
polygonal stadium such that \eqref{contra1} holds.\par {\em Second
case:} Assume that
\begin{equation}\label{contra2}
\frac{\wt(\Om)|\partial\Om|^q}{|\Om|^{q+1}}\leq\frac{1}{q+1}.
\end{equation}
Hence,
$$\left[(q+1)\frac{|\partial\Om|^{q+1}}{|\Om|^{q+2}}\wt(\Om)-\frac{|\partial\Om|}{|\Om|}-
2q\frac{C_{\Om}\wt(\Om)|\partial\Om|^{q-1}}{|\Om|^{q+1}}\right]$$
$$=(q+1)\frac{|\partial\Om|}{|\Om|}\left[\frac{|\partial\Om|^q}{|\Om|^{q+1}}\wt(\Om)-\frac{1}{q+1}-
\frac{2q}{q+1}\frac{C_{\Om}\wt(\Om)|\partial\Om|^{q-2}}{|\Om|^q}\right]<0\,
.$$ Inserting this into \eqref{comparison} and arguing as in the
previous case, we see that the map
$t\mapsto\frac{\wt(\Om_t)|\partial\Om_t|^q}{|\Om_t|^{q+1}}$ is
strictly decreasing for $t\in[0,R_\Omega)$. In view of Proposition
\ref{becomestadium}, this proves that there exists some polygonal
stadium  such that \eqref{contra2} holds.\par\smallskip\noindent
{\it $\bullet$ Step 3: explicit computation for a polygonal stadium.} Let
$\Omega= P ^\ell\in\S$ be a polygonal stadium. We are going to
derive an explicit expression for the function
$$F(\ell):=\frac{\wt(P^{\ell})|\partial P^{\ell}|^q}{|P^{\ell}|^{q+1}}\qquad \forall \, \ell\ge0\, .$$
We point out that, in the special case $\ell = 0$, $\Omega \in \P \cap \C_o$ (namely $\Om$ is a circumscribed polygon), and
it is proven in \cite[Proposition 2]{cfg1} that
\begin{equation}\label{eq:energycirc}
\forall \Om\in\C_o,\quad\frac{\wt(\Om)|\partial\Om|^q}{|\Om|^{q+1}}=\frac{2}{q+2}.
\end{equation}
In particular, formula \eq{eq:energycirc} shows that the upper bound in \eqref{eq:isoNp} is achieved when $\Omega\in\C_o$.\par
We now show that the above formula can be suitably extended also to the case $\ell>0$. Our starting point is the representation formula (\ref{ice}).
Therein, we use the Steiner formulae \eqref{steiner}; in particular, by Propositions \ref{boh} and \ref{speciali},
we know
that $C_{\Omega_t}\equiv C_\Omega$ for every $t\in[0,R_\Omega)$.
Moreover, since $P \in \P \cap \C _o$, we can exploit identities (\ref{circum}). Setting for brevity
$$A:= | P|\ , \qquad R := R_P\ , \qquad x:= \frac{2 R \ell} {A}\ ,$$
we obtain
\begin{eqnarray}
F(\ell)&=&\frac{\left(2\frac{A}{R}+2\ell\right)^q}{\left(A+2R\ell\right)^{q+1}}
\int_{0}^R\frac{\left(A+2R\ell-2\ell t-2\frac{A}{R}t+\frac{A}{R^2}t^2\right)^q}{\left(2\ell+2\frac{A}{R}-2\frac{A}{R^2}t\right)^{q-1}}dt\,
 \notag\\
&=&\frac{(x+2)^q}{(x+1)^{q+1}}\int_0^1\frac{(1+x-xt-2t+t^2)^q}{(x+2-2t)^{q-1}}\, dt \notag\\
&=&\frac{(x+2)^q}{(x+1)^{q+1}}\int_0^1\; \frac{t^q(x+t)^q}{(x+2t)^{q-1}}\, dt\ .\label{FF}
\end{eqnarray}
Of course, taking $x=0$ in \eqref{FF} gives again
\eqref{eq:energycirc}; on the other hand, taking $x\to\infty$ gives
the asymptotic behaviour for thinning polygonal stadiums.\par\smallskip\noindent
{\it $\bullet$ Step 4}: In view of equality \eqref{FF} obtained in Step 3, the estimate
\eqref{eq:isoNp} will be proved for any polygonal stadium, provided
we show that for all $q\in(1,+\infty)$ one has
\neweq{double}
\frac{1}{q+1}<\frac{(x+2)^q}{(x+1)^{q+1}}\,
\int_0^1\frac{t^q(x+t)^q}{(x+2t)^{q-1}}\,
dt<\frac{2}{q+2}\qquad\forall x\in(0,+\infty).
\endeq
With the change of variables $t=xs$, the inequalities in \eq{double} become
$$\frac{1}{q+1}\frac{(x+1)^{q+1}}{x^{q+2}(x+2)^q}<\int_0^{1/x}\frac{s^q(1+s)^q}{(1+2s)^{q-1}}\, ds<
\frac{2}{q+2}\, \frac{(x+1)^{q+1}}{x^{q+2}(x+2)^q}\qquad\forall x\in(0,+\infty).$$
In turn, by putting $y=1/x$, the latter inequalities become
\neweq{double2}
\frac{1}{q+1}\frac{y^{q+1}(1+y)^{q+1}}{(1+2y)^q}<\int_0^y
\frac{s^q(1+s)^q}{(1+2s)^{q-1}}\, ds< \frac{2}{q+2}\,
\frac{y^{q+1}(1+y)^{q+1}}{(1+2y)^q}\qquad\forall y\in(0,+\infty).
\endeq
In order to prove the right inequality in \eq{double2}, consider the function
$$\Phi(y):=\int_0^y \frac{s^q(1+s)^q}{(1+2s)^{q-1}}\, ds-\frac{2}{q+2}\, \frac{y^{q+1}(1+y)^{q+1}}{(1+2y)^q}\qquad y\in(0,+\infty)$$
and we need to prove that $\Phi(y)<0$ for all $y>0$. This is a consequence
of the two following facts:
$$\Phi(0)=0\ ,\qquad\Phi'(y)=-\frac{q}{q+2}\frac{y^q(1+y)^q}{(1+2y)^{q+1}}<0\, .$$
In order to prove the left inequality in \eq{double2}, consider the function
$$\Psi(y):=\int_0^y \frac{s^q(1+s)^q}{(1+2s)^{q-1}}\, ds-\frac{1}{q+1}\, \frac{y^{q+1}(1+y)^{q+1}}{(1+2y)^q}\qquad y\in(0,+\infty)$$
and we need to prove that $\Psi(y)>0$ for all $y>0$. This is a consequence of the two following facts:
$$\Psi(0)=0\ ,\qquad\Psi'(y)=\frac{2q}{q+1}\frac{y^{q+1}(1+y)^{q+1}}{(1+2y)^{q+1}}>0\, .$$
Both inequalities in \eq{double2} are proved and \eq{double} follows.
\par\smallskip\noindent
We point out that, in the case $q=2$, some explicit computations
give the stronger result that the map
$$x\mapsto \frac{(x+2)^q}{(x+1)^{q+1}}\, \int_0^1\frac{t^q(x+t)^q}{(x+2t)^{q-1}}\, dt$$
is decreasing. We believe that this is true for any $q$, but we do
not have a simple proof of this property.\par\smallskip\noindent
{\it $\bullet$ Step 5: conclusion.} Let $\Omega\in\P$ and assume for
contradiction that $\Omega$ violates \eq{eq:isoNp}. Then by Step 2
we know that there exists a polygonal stadium which also violates
\eqref{eq:isoNp}. This contradicts Step 4, see \eq{double}. We have
so far proved that \eqref{eq:isoNp} holds for all $\Omega\in\P$.
By a density argument we then infer that
\begin{equation}\label{last}
\forall \Om\in\C,\quad \frac{1}{q+1}\le\frac{\wt(\Om)|\partial\Om|^q}{|\Om|^{q+1}}\leq\frac{2}{q+2}.
\end{equation}

Therefore, in order to complete the proof we need to show that the left inequality in \eqref{last} is strict.
Assume for contradiction that there exists $\Om\in\C$ such that
\begin{equation}\label{contradiction}
\frac{\wt(\Om)|\partial\Om|^q}{|\Om|^{q+1}}=\frac{1}{q+1}\, .
\end{equation}
Take any sequence $\Om^k\in\P$ such that $\Om^k\supset\Omega$ and $\Om^k\to\Omega$ in the Hausdorff topology.
Similar computations as in \eqref{discuss}, combined with \eqref{macrosteiner}, enable us to obtain
$$
\frac{\wt(\Om^k_{\eps})|\partial\Om^k_{\eps}|^q}{|\Om^k_{\eps}|^{q+1}}\le
\frac{\left[\wt(\Om^k)-\int_{0}^{\eps}\frac{|\Om^k_{t}|^q}{|\partial\Om^k_{t}|^{q-1}}dt\right]\left[|\partial\Om^k|-2\pi\, \eps\right]^q}
{\left[|\Om^k|-|\partial\Om^k|\, \eps\right]^{q+1}}$$
$$\le\frac{\wt(\Om^k)|\partial\Om^k|^q}{|\Om^k|^{q+1}}+\left[(q+1)\frac{|\partial\Om^k|^{q+1}}{|\Om^k|^{q+2}}\wt(\Om^k)-
\frac{|\partial\Om^k|}{|\Om^k|}-2q\frac{\pi\wt(\Om^k)|\partial\Om^k|^{q-1}}{|\Om^k|^{q+1}}\right]\eps+\alpha\eps^2,$$
where $\alpha$ is some positive constant, depending on $\Om$ but not on $k$. Therefore,
since $\Om^k_t\to\Omega_t$ for all $t\in[0,R_\Omega]$, we have
\begin{eqnarray*}
\frac{\wt(\Om_{\eps})|\partial\Om_{\eps}|^q}{|\Om_{\eps}|^{q+1}}-\frac{\wt(\Om)|\partial\Om|^q}{|\Om|^{q+1}} &=&
\frac{\wt(\Om^k_{\eps})|\partial\Om^k_{\eps}|^q}{|\Om^k_{\eps}|^{q+1}}-\frac{\wt(\Om^k)|\partial\Om^k|^q}{|\Om^k|^{q+1}}+o(1)\\
&\le&\left[o(1)-2q\frac{\pi\wt(\Om^k)|\partial\Om^k|^{q-1}}{|\Om^k|^{q+1}}\right]\eps+\alpha\eps^2+o(1)\\
\end{eqnarray*}
where $o(1)$ are infinitesimals (independent of $\eps$) as $k\to\infty$.
Hence, by letting $k\to\infty$ and taking $\eps$ sufficiently small, we obtain
$\frac{\wt(\Om_{\eps})|\partial\Om_{\eps}|^q}{|\Om_{\eps}|^{q+1}}<\frac{1}{q+1}$, which contradicts \eqref{last}.

\subsection{Proof of Theorem \ref{estimate2}}

The inequalities \eqref{eq:isoDp} follow directly from
\eqref{eq:isoNp} and \eqref{proven} so we just need to show that
they are sharp.\par For the right inequality, take a sequence of
thinning isosceles triangles $T_k$. Then, by Theorem \ref{estimate} we have
$$\frac{\wt(T_k)|\partial T_k|^q}{|T_k|^{q+1}}=\frac{2}{q+2}\qquad\mbox{for all }k\, .$$
On the other hand, by \cite[Proposition 3]{cfg1} and \eqref{proven}
we know that
$$\lim_{k\to\infty}\frac{\wt(T_k)}{\t(T_k)}=\frac{q+1}{2^q}$$
and therefore
$$\lim_{k\to\infty}\frac{\t(T_k)|\partial T_k|^q}{|T_k|^{q+1}}=\frac{2^{q+1}}{(q+2)(q+1)}\, .$$
For the left inequality, we seek an upper bound for $\t(\Omega)$ by
using the maximum principle. For all $\ell\in(0,+\infty)$ let
$\Omega^\ell=(-\frac{\ell}{2},\frac{\ell}{2})\times(-1,1)$ and let
$u_\ell$ be the unique solution to
$$-\Delta_p u_\ell=1\quad\mbox{in }\Omega^\ell\ ,\qquad u_\ell=0\quad\mbox{on }\partial\Omega^\ell\ .$$
Let $u_\infty(x,y)=\frac{p-1}{p}(1-|y|^{p/(p-1)})$ so that
$$-\Delta_p u_\infty=1\quad\mbox{in }\Omega^\ell\ ,\qquad u_\infty\ge0\quad\mbox{on }\partial\Omega^\ell\ .$$
By the maximum principle, we infer that $u_\infty\ge u_\ell$ in
$\Omega^\ell$ so that
$$\t(\Omega^\ell)=\int_{\Omega^\ell}u_\ell\le\int_{\Omega^\ell}u_\infty=\frac{2(p-1)}{p}\ell\int_0^1(1-y^{p/(p-1)})\, dy=\frac{2(p-1)}{2p-1}\ell=\frac{2\ell}{q+1}\, .$$
Hence,
$$1\ge\liminf_{\ell\to\infty}\frac{\wt(\Omega^\ell)}{\t(\Om^\ell)}\ge\liminf_{\ell\to\infty}\frac{(q+1)\wt(\Omega^\ell)}{2\ell}=1$$
where the last equality follows from Theorem \ref{estimate}. Combined with Theorem \ref{estimate}, this proves that
$$\lim_{\ell\to\infty}\frac{\t(\Om^\ell)|\partial\Om^\ell|^q}{|\Om^\ell|^{q+1}}=\frac{1}{q+1}\, .$$

\subsection{Proof of Theorem \ref{new}}

Since it follows closely the proof of Theorem \ref{estimate2}, we just sketch it. We first prove the counterpart of Theorem \ref{estimate}
and we follow the same steps.\par
{\it $\bullet$ Step 1.} Given $\Om\in\P$ and using $R_{\Om_{\eps}}=R_{\Om}-\eps+o(\eps)$ we prove:
\begin{equation}\label{eq:newparallel}
\frac{\wt(\Om_{\eps})}{R_{\Om_{\eps}}|\Om_{\eps}|}-
\frac{\wt(\Om)}{R_{\Om}|\Om|}=\frac{\eps}{R_{\Om}^q|\Om|}\left(w_{p}(\Om)\left[\frac{q}{R_{\Om}}+
\frac{|\partial\Om|}{|\Om|}\right]-\frac{|\Om|^{q}}{|\partial\Om|^{q-1}}\right)+o(\eps)\, .
\end{equation}
\par
{\it $\bullet$ Step 2.} We prove that, if \eqref{new1} fails for some $\Om\in\P$, then it also fails for a polygonal stadium. To that end, we estimate the sign in \eqref{eq:newparallel} with the help of the following classical geometric inequalities (see \cite{bonn})
$$\forall \Om\in\C, \;\;\frac{|\Om|}{R_{\Om}}<|\partial\Om|\leq\frac{2|\Om|}{R_{\Om}}\;.$$
\par
{\it $\bullet$ Step 3.} Again, explicit computations can be done for a polygonal stadium, and with the same notation as in the proof of Theorem \ref{estimate}, we get:
$$\frac{w_{p}(P^\ell)}{R_{P^\ell}^q|P^\ell|}=\frac{1}{x+1}\int_0^1\frac{t^q(x+t)^q}{(x+2t)^{q-1}}\,dt \qquad 
\forall P ^ \ell \in {\mathcal S}\,.$$
\par
{\it $\bullet$ Step 4.} In view of Step 3, estimate (\ref{new2}) is proved for any polygonal stadium, provided for all $q\in(1,+\infty)$ one has
\neweq{double3}
\frac{1}{(q+2)2^{q-1}}<\frac{1}{x+1}\, \int_0^1\frac{t^q(x+t)^q}{(x+2t)^{q-1}}\,
dt<\frac{1}{q+1}\qquad\forall x\in(0,+\infty).
\endeq
With the change of variables $t=xs$ and putting $y=1/x$, the inequalities in \eq{double3} become
\neweq{double4}
\frac{y^{q+2}+y^{q+1}}{(q+2)2^{q-1}}<\int_0^y\frac{s^q(1+s)^q}{(1+2s)^{q-1}}\,
ds<\frac{y^{q+2}+y^{q+1}}{q+1}\qquad\forall y\in(0,+\infty).
\endeq
Some tedious but straightforward computations show that
$$
\frac{y^{q+1}}{2^{q-1}}+\frac{q+1}{(q+2)2^{q-1}}y^q<\frac{y^q(1+y)^q}{(1+2y)^{q-1}}<\frac{q+2}{q+1}y^{q+1}+y^q
\qquad\forall y\in(0,+\infty)
$$
and \eq{double4} follows after integration over $(0,y)$.\par
{\it $\bullet$ Step 5.} The previous steps leads to \eqref{new1} for polygons and by density for convex domains. The strict right inequality in \eqref{new1} can be obtained by reproducing carefully the computations in Step 1, similarly as done in Step 5 of Section \ref{ssect:proof1}. 

Now the counterpart of Theorem \ref{estimate} is proved, and we may
use \eqref{proven} in order to get \eqref{new2} from \eqref{new1}. Balls realize equality in the left inequality of \eqref{new2} because the are at the same time circumscribed and maximal for the quotient $w_{p}/\t$.

\section{Some open problems}\label{open}

We briefly suggest here some perspectives which might be considered, in the light of our results.\par\smallskip
{\em Sharp bounds for the $p$-torsion in higher dimensions.} In higher dimensions the shape functionals $\t$ and $\wt$ can be
defined in the analogous way as for $n=2$. In \cite{cr2}, Crasta proved the following sharp bounds:
$$\forall\; \Om\textrm{ bounded convex set }\subset\R^n\,, \qquad \frac{n+1}{2n}<\frac{w_2(\Om)}{\tau_{2}(\Om)}\leq 1\ .$$
Therefore it seems natural to ask: what kind of isoperimetric inequality can be proved for $w_{p}$  and $\tau _p$ among convex sets in $\R^n$?
In this direction, let us quote an inequality proven in \cite{FK}, obtained by a strategy similar to our approach, that is by looking at
the level sets of the support function:
$$\forall\; \Om\textrm{ bounded convex set }\subset\R^n\,, \qquad
\frac{\tau_{2}(\Om)|\partial\Om|}{R_{\Om}|\Om|^2} \geq \frac{\tau_{2}(B)|\partial B|}{R_{B}|B|^2} \qquad(B\mbox{ is a ball of }\R^n\, ).$$
{\em Sharp bounds for the principal frequency.}
A notion of ``web principal frequency'' can be defined (in any space dimension) similarly as done for the web torsion, that is
$$\lambda_{1}^+(\Om):=\inf\left\{\frac{\int_{\Om}|\nabla u|^2}{\int_{\Om}u^2}\ :\  u\in \mathcal{W}_{2}(\Om)\right\}.$$
Writing the optimality condition in the space $\mathcal{W}_{2}(\Om)$, one can express $\lambda_{1}^+(\Om)$ as
$$\lambda_{1}^+(\Om)=\inf\left\{\frac{\int_{0}^{R_{\Om}}\alpha\rho'^2}{\int_{0}^{R_{\Om}}\alpha\rho^2}\ :\ \rho \in H ^ 1 (0, R_\Om)\, ,\   \rho(0)=0\right\}, \textrm{ where }\alpha(t)=|\partial\Om_{t}|.$$
It is clear that $\lambda_{1}^+ (\Om) \geq \lambda_{1}(\Om)$, with
equality sign when $\Om$ is a ball. On the other hand, the following questions can be addressed:\par
$\bullet$ Find a sharp bound from above for the ratio $\lambda_{1}^+(\Om)/\lambda_{1}(\Om)$ among bounded convex subsets of $\R^n$.\par
$\bullet$ Is it possible to apply successfully the same strategy of this paper, that is find sharp bounds for $\lambda_{1}^+(\Om)$
and then use the estimates on the ratio $\lambda_{1}^+(\Om)/\lambda_{1}(\Om)$, to deduce sharp bounds for $\lambda_{1}(\Om)$?
In particular, this approach might allow to retrieve the following known inequalities holding for any bounded convex domain
$\Omega \subset \R^2$ (see \cite{he,makai,polya}):
$$\displaystyle{\frac{\pi^2}{16}\leq \lambda_{1}(\Om)\frac{|\Om|^2}{|\partial\Om|^2}\leq \frac{\pi^2}{4}} \qquad \hbox{ and } \qquad \displaystyle{\frac{\pi^2}{4}\leq \lambda_{1}(\Om)R_{\Om}^2\leq j_{0}^2\ .}$$

\bigskip
{\bf Acknowledgments.} This work was realized thanks to the
financial support of the Italian Gruppo Nazionale per l'Analisi
Matematica, la Probabilit\`a e le loro Applicazioni,  as part of the
project ``Problemi di ottimizzazione e disuguaglianze
geometrico-funzionali'', and by the French Agence Nationale de la
Recherche, as part of the project ANR-09-BLAN-0037 ``Geometric
analysis of optimal shapes''.

\par\smallskip
\small{ Ilaria FRAGAL\`A, Politecnico di Milano, Piazza Leonardo da Vinci
32, 20133 Milano (Italy)\par\medskip Filippo GAZZOLA, Politecnico di
Milano, Piazza Leonardo da Vinci 32, 20133 Milano (Italy)
\par\medskip Jimmy LAMBOLEY, Universit\'e Paris-Dauphine,
Ceremade,  Place du Mar\'echal de Lattre de Tassigny, 75775 Paris (France)}
\vfill\eject
\end{document}